\documentclass[11pt,letterpaper,twoside,english]{article}
\usepackage[margin=1.5in]{geometry}

\usepackage{mdwlist}

\usepackage{amssymb,amsbsy,latexsym}
\usepackage{amsmath}
\usepackage{graphics, subfigure, float} 
\usepackage{url}

\usepackage[latin1]{inputenc}
\usepackage[american]{babel}
\usepackage[T1]{fontenc} 
\usepackage{fourier}

\usepackage{amscd,amsthm}

\usepackage{verbatim, comment}

\usepackage[dvips]{graphicx,epsfig,color}
\usepackage{pst-all}
\usepackage{pstricks-add}
\usepackage{fp,calc}
\usepackage{datetime}

\newtheoremstyle{theorem}{1em}{1em}{\slshape}{0pt}{\bfseries}{.}{ }{}
\theoremstyle{theorem}
\newtheorem{theorem}{Theorem}
\newtheorem*{theorem*}{Theorem}

\newtheorem{lemma}[theorem]{Lemma}
\newtheorem{claim}[theorem]{Claim}

\newtheorem{definition}[theorem]{Definition}

\theoremstyle{remark}

\newtheorem*{remark*}{Remark}

\newtheorem*{question*}{Question}

\providecommand{\setZ}{\mathbb{Z}}

\providecommand{\setR}{\mathbb{R}}
\newcommand{\conv}{\textrm{conv}}

\newcommand{\supp}{\textrm{supp}}
\newcommand{\rk}{\textrm{rk}}

\floatstyle{ruled}
\newfloat{algorithm}{tbp}{loa}
\floatname{algorithm}{Algorithm}

\newenvironment{proofofclaim}{\vspace{1ex}\noindent{\emph{Proof of claim.}}\hspace{0.5em}}
   	    {\hfill$\lozenge$\vspace{1ex}}

        \def\drawRect#1#2#3#4#5{
           \FPeval{\x2}{(#2) + (#4)} 
           \FPeval{\y2}{(#3) + (#5)} 
           \pspolygon[#1](#2,#3)(\x2,#3)(\x2,\y2)(#2,\y2)
        }

\date{\today} 
\makeatother

\begin{document}
\title{0/1 Polytopes with Quadratic Chvátal Rank} 

\author{Thomas Rothvoß\thanks{M.I.T. Email: { \tt{rothvoss@math.mit.edu}}. Supported by the Alexander von Humboldt Foundation within the Feodor Lynen program, by ONR grant N00014-11-1-0053 and by NSF contract
CCF-0829878.}  \and  Laura Sanità\thanks{University of Waterloo. Email: {\tt lsanita@uwaterloo.ca}}  }

\maketitle


\begin{abstract}
For a polytope $P$, the \emph{Chvátal closure} $P' \subseteq P$ is obtained by simultaneously
strengthening all feasible inequalities $cx \leq \beta$ (with integral $c$) to $cx \leq \lfloor\beta\rfloor$.
The number of iterations of this procedure that are needed until the integral hull of $P$
is reached is called the \emph{Chvátal rank}. 
If $P \subseteq [0,1]^n$, then it is known that $O(n^2 \log n)$ iterations
always suffice (Eisenbrand and Schulz (1999)) and at least $(1+\frac{1}{e}-o(1))n$ 
iterations are sometimes needed (Pokutta and Stauffer (2011)), 
leaving a huge gap between lower and upper bounds.

We prove that there is a polytope contained in the 0/1 cube that has Chvátal rank $\Omega(n^2)$, 
closing the gap up to a logarithmic factor. 
In fact, even a superlinear lower bound was mentioned as an open problem by several authors.
Our choice of $P$ is the convex hull of a 
semi-random Knapsack polytope and a single fractional vertex. The main technical ingredient
is linking the Chvátal rank to \emph{simultaneous Diophantine approximations} w.r.t. the $\|\cdot\|_1$-norm
of the normal vector defining $P$. 
\end{abstract}

\section{Introduction}

Gomory-Chvátal cuts are among the most important classes of cutting planes 
used to derive the integral hull of polyhedra.
The fundamental idea to derive such cuts is that if an inequality $cx \leq \beta$ is \emph{valid} for a polytope $P$
(that is, $cx \leq \beta$ holds for every $x \in P$) and $c \in \setZ^n$, then $cx \leq \lfloor\beta\rfloor$ is valid for 
the integral hull $P_I:= \conv( P \cap \setZ^n )$.
Formally, for a polytope $P \subseteq \setR^n$ and a vector $c \in \setZ^n$,
\[
  GC_P(c) := \big\{ x \in \setR^n \mid cx \leq \lfloor\max\{ cy \mid y \in P \}\rfloor  \big\} 
\]
is the \emph{Gomory-Chvátal Cut} that is induced by vector $c$ (for polytope $P$). 
Furthermore,
\[
  P' := \bigcap_{c \in \setZ^n} GC_P(c)
\]
is the \emph{Gomory-Chvátal closure} of $P$. 

Let $P^{(i)} := (P^{(i-1)})'$ (and $P^{(0)} = P$)
be the \emph{$i$th Gomory-Chvátal closure} of $P$. The \emph{Chvátal rank} $\rk(P)$ is the 
smallest number such that $P^{(\rk(P))} = P_I$.


It is well-known that the Chvátal rank is always finite, but can be arbitrarily large already for 2 dimensional polytopes.
However, if we restrict our attention to polytopes $P \subseteq [0,1]^n$ contained in the $0/1$ cube 
the situation becomes much different, since the  Chvátal rank can be bounded by a function
in $n$. In particular, Bockmayr, Eisenbrand, Hartmann and Schulz~\cite{DBLP:journals/dam/BockmayrEHS99} 
provided the first polynomial upper bound of $\rk(P)\leq O(n^3 \log n)$.
Later, Eisenbrand and Schulz \cite{ChvatalRankEisenbrandSchulzIPCO99,ChvatalRankEisenbrandSchulzCombinatorica03}
proved that $\rk(P) \leq O(n^2 \log n)$, which is still the best known upper bound. 
Note that if $P \subseteq [0,1]^n$ and $P \cap \{ 0,1\}^n = \emptyset$, then even $\rk(P) \leq n$ (and this is
tight if and only if $P$ intersects all the edges of the hypercube~\cite{EmptyPolytopesWithGCRankN-PokuttaSchulzORL11}).
Already \cite{On-cutting-Plane-proofs-ChvatalCookHartmann1989} provided lower bounds on the rank
for the polytopes corresponding to natural problems like stable-set, set-covering, set-partitioning, knapsack,
maxcut and ATSP (however, none of the bounds exceeded $n$). 
The paper of Eisenbrand and Schulz~\cite{ChvatalRankEisenbrandSchulzIPCO99,ChvatalRankEisenbrandSchulzCombinatorica03} also provides a lower bound $\rk(P) > (1 + \varepsilon)n$ for a tiny constant $\varepsilon >0$, which has been 
quite recently improved by
 Pokutta and Stauffer~\cite{ChvatalLowerBounds-PokuttaStaufferORL11}
 to $(1+\frac{1}{e}-o(1))n$. 
However, as the authors of \cite{EmptyPolytopesWithGCRankN-PokuttaSchulzORL11} state, there is still a 
very large gap between the best known upper and lower bound.
In particular, the question whether there is any \emph{superlinear} lower bound on the rank
of a polytope in the $0/1$ cube is open since many years (see e.g.  
Ziegler~\cite{0-1-PolytopesZiegler97}).

There is a large amount of results on structural properties of the CG closure. 
Already Schrijver~\cite{OnCuttingPlanes-Schrivjer1980} could prove that that the closure
of a rational polyhedron is again described by finitely many inequalities.
 Dadush, Dey and Vielma~\cite{CGclosureOfStrictlyConvexBodyDadushDeyVielmaMOR11} showed that $K'$ is a polytope for all compact and strictly convex sets $K \subseteq \setR^n$. 
Later,  Dunkel and Schulz~\cite{GCclosure-of-non-rational-polytopes-DunkelSchulz2011} could prove the same if $K$ is an irrational polytope, while in parallel
again Dadush, Dey and Vielma~\cite{CGclosureOfConvexSet-DadushDeyVielmaIPCO2011} showed 
that this holds in fact for \emph{any} compact convex set.

In the last years, automatic procedures that strengthen existing relaxations became more and more
popular in theoretical computer science. 
Singh and Talwar~\cite{ImprovingIntegralityGapsViaCG-SinghTalwarAPPROX10} showed that few CG rounds reduce the integrality gap for $k$-uniform hypergraph matchings.
However, to obtain approximation algorithms researchers rely more on \emph{Lift-and-Project Methods}
such as the hierarchies of
\emph{Balas, Ceria, Cornuéjols}~\cite{BalasCeriaCornuejols-Hierarchy-MathProg93}; 
\emph{Lovász, Schrijver}~\cite{LovaszSchrijverHierarchy91}; 
\emph{Sherali, Adams}~\cite{SheraliAdamsHierarchy1990}
or \emph{Lasserre}~\cite{ExplicitExactSDP-Lasserre-IPCO01,GlobalOpt-Lasserre01}.
One can optimize over the $t$th level in time $n^{O(t)}$. 
Moreover, all those hierarchies converge to the integral hull already after $n$ iterations.
In contrast, the membership problem for $P'$ is $\mathbf{coNP}$-hard~\cite{CG-membership-hard-Eisenbrand99}. 
We refer to the surveys of Laurent~\cite{SDP-hierarchies-Survey-Laurent2003} and Chlamtá\v c and Tulsiani~\cite{ConvexRelaxations-survey-Chlamtac-Tulsiani}
for a detailed  comparison.


In this paper, we prove that there is a polytope contained in the $0/1$ cube that has Chvátal rank $\Omega(n^2)$, 
closing the gap up to a logarithmic factor. Specifically, our main result is:
\begin{theorem}
For every $n$, there exists a vector $c \in \{ 0,\ldots,2^{n/16}\}^n$ such that
the polytope
\[
 P = \conv\Big\{\Big\{ x \in \{ 0,1\}^n : \sum_{i=1}^n c_ix_i \leq \frac{\|c\|_1}{2} \Big\} \cup \Big\{ \Big(\frac{3}{4},\ldots,\frac{3}{4}\Big) \Big\} \Big\} \subseteq [0,1]^n
\]
has Chvátal rank $\Omega(n^2)$.
\end{theorem}

Here $\|c\|_1 := \sum_{i=1}^n |c_i|$ and $\|c\|_{\infty} := \max_{i=1,\ldots,n} |c_i|$.

\section{Outline}

In the following, we provide an informal outline of our approach.
\begin{enumerate}
\item[(1)] \emph{The polytope.} Our main result is to show that the polytope
\[
   P(c,\varepsilon) := \conv\Big\{ \Big\{ x \in \{ 0,1\}^n :  cx \leq \frac{\|c\|_1}{2} \Big\} \cup \{ x^*(\varepsilon) \} \Big\}
\]
has a Chvátal rank of $\Omega(n^2)$, where $x^* := x^*(\varepsilon) := (\frac{1}{2} + \varepsilon,\ldots,\frac{1}{2} + \varepsilon)$ (see Figure~\ref{fig:PolytopeP}.(a)).
We can choose $\varepsilon := \frac{1}{4}$ and each $c_i$ will be an integral coefficient
of order $2^{\Theta(n)}$ --- however, we postpone the precise choice of $c$ for now.
Intuitively spoken,  $P$ is a Knapsack polytope defined by inequality $cx \leq \frac{\|c\|_1}{2}$
plus an extra fractional vertex $x^*$. Observe that the vector $x^*(0) = (\frac{1}{2},\ldots,\frac{1}{2})$
satisfies that constraint with equality. 
\begin{figure}
\begin{center}
\psset{unit=4cm}
\subfigure[]{ 
\begin{pspicture}(0.2,-0.25)(1.7,1.15)
\pspolygon[linestyle=dashed,linewidth=0.75pt](0,0)(0,1)(1,1)(1,0)
\pspolygon[fillstyle=solid,fillcolor=lightgray,linewidth=0.75pt](0,0)(0,1)(0.7,0.7)(1,0)
\cnode*(0,0){2.5pt}{x00}
\cnode*(1,0){2.5pt}{x10}
\cnode*(0,1){2.5pt}{x01}
\cnode*(1,1){2.5pt}{x11}
\psline[linewidth=1.5pt](0,1)(1,0) 
\rput[c](0.4,0.2){$P$}
\cnode*(0.5,0.5){2.5pt}{z} \nput[labelsep=0pt]{-135}{z}{$(\frac{1}{2},\ldots,\frac{1}{2})$}
\pnode(0.9,0.9){c}
\cnode*(0.7,0.7){2.5pt}{y} \nput{0}{y}{$x^* = (\frac{1}{2}+\varepsilon,\ldots,\frac{1}{2}+ \varepsilon)$}
\ncline[linestyle=dashed]{->}{z}{c} 
\psline[linewidth=1.5pt](-0.2,1.2)(1.2,-0.2) \rput[l](1.15,-0.1){$cx\leq\|c\|_1/2$}

\end{pspicture}
}
\subfigure[]{
\begin{pspicture}(-0.1,-0.25)(1.5,1.15)
\pspolygon[linestyle=dashed,linewidth=0.75pt](0,0)(0,1)(1,1)(1,0)
\pspolygon[fillstyle=solid,fillcolor=lightgray,linewidth=0.75pt](0,0)(0,1)(0.8,0.8)(1,0)
\pspolygon[fillstyle=solid,fillcolor=gray,linewidth=0.75pt](0,0)(0,1)(0.7,0.7)(1,0)
\cnode*(0,0){2.5pt}{x00}
\cnode*(1,0){2.5pt}{x10}
\cnode*(0,1){2.5pt}{x01}
\cnode*(1,1){2.5pt}{x11}
\psline[linewidth=1.5pt](0,1)(1,0) 
\rput[c](0.7,0.45){$P'$}
\rput[c](0.85,0.5){$P$}
\cnode*(0.5,0.5){2.5pt}{z} \nput[labelsep=0pt]{-135}{z}{$(\frac{1}{2},\ldots,\frac{1}{2})$}
\pnode(0.9,0.9){c}
\cnode*(0.8,0.8){2.5pt}{xStar} 
\cnode*(0.7,0.7){2.5pt}{xPrime} 
\pnode(1.1,0.9){xStarLabel} \ncline[arrowsize=5pt]{->}{xStarLabel}{xStar}  \nput[labelsep=3pt]{0}{xStarLabel}{$x^*(\varepsilon)$} 
\pnode(0.7,1.1){xPrimeLabel} \ncline[arrowsize=5pt]{->}{xPrimeLabel}{xPrime}  \nput[labelsep=2pt]{90}{xPrimeLabel}{$x^*(\varepsilon')$} 
\psline[linewidth=2pt](0,1.1)(0.7,0.7)(1.4,0.3) \rput[l](1.43,0.3){$\tilde{c}x \leq \lfloor\beta\rfloor$}
\psline[linewidth=2pt](0.1,1.2)(0.8,0.8)(1.5,0.4) \rput[l](1.53,0.4){$\tilde{c}x \leq \beta$}
\ncline[linestyle=dashed]{->}{z}{c} 

\end{pspicture}
}
\caption{(a) Polytope $P=P(c,\varepsilon)$ in $n=2$ dimensions and with $c=(1,1)$. (b) Visualization of the  Gomory Chvátal cut $\tilde{c}x \leq \beta$ 
for a critical vector $\tilde{c}$. 
Note that $\max\{ \tilde{c}x \mid x \in P\} = \tilde{c}x^*(\varepsilon)$.\label{fig:PolytopeP}}
\end{center}
\end{figure}
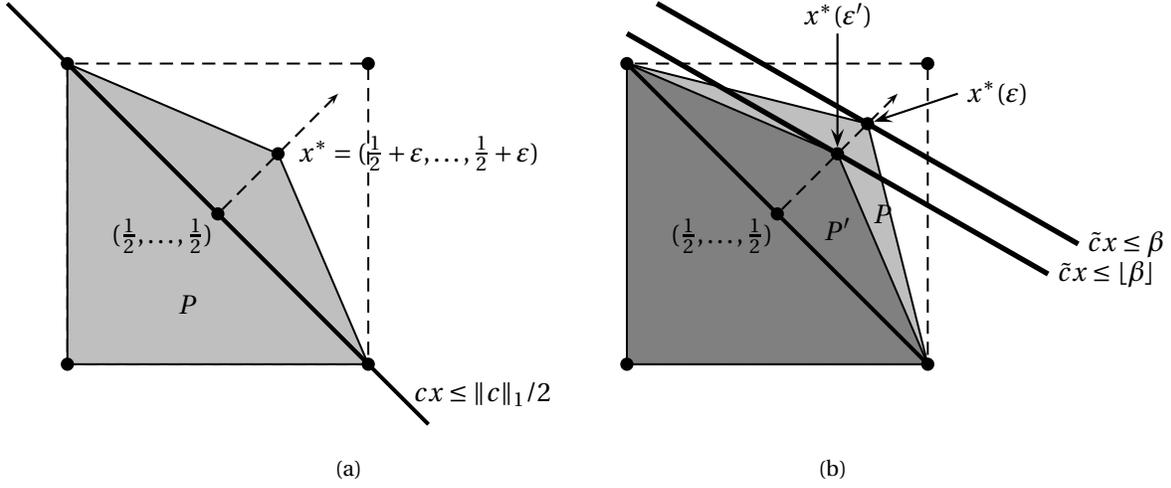
\item[(2)] \emph{The progress of the GC operator.} We will measure the progress of the Gomory Chvátal operator by observing how much
of the line segment between $\frac{1}{2} \mathbf{1}$ and $\frac{3}{4} \mathbf{1}$ has been
cut off. Consider a single Gomory Chvátal round and that Chvátal cut $\tilde{c}x \leq \lfloor\beta\rfloor$ that cuts off 
the longest piece from the line segment. In other words, $\tilde{c}x \leq \beta$ is valid for $P$, but $\tilde{c}x^* > \lfloor\beta\rfloor$.
Of course, a necessary condition on such a vector $\tilde{c}$ is that the objective 
function $\tilde{c}$ is maximized at $x^*$. Let us call any such a vector \emph{critical} (see Figure~\ref{fig:PolytopeP}.(b)).
Secondly, the point $x^*(\varepsilon') \in P'$ with maximum $\varepsilon'$ must have $\tilde{c}x^*(\varepsilon') = \lfloor\beta\rfloor$.
But that means
\[
  1 \geq \tilde{c}x^*(\varepsilon) - \tilde{c}x^*(\varepsilon') = \tilde{c}\mathbf{1} \cdot (\varepsilon - \varepsilon') = \|\tilde{c}\|_1 \cdot (\varepsilon - \varepsilon')
\]
and we can bound the progress of the Gomory Chvátal operator by $\varepsilon-\varepsilon' \leq \frac{1}{\|\tilde{c}\|_1}$.
In other words, in order to show a high rank, we need to prove that
all critical vectors must be long.

We will later propose a choice of $c$ such that any critical vector $\tilde{c}$
has $\|\tilde{c}\|_1 \geq \Omega(\frac{n}{\varepsilon})$ (as long as  $\varepsilon \geq (\frac{1}{2})^{O(n)}$). 
This means that the number of GC iterations
until the current value of $\varepsilon$ reduces to $\varepsilon/2$ will be $\Omega(n)$; thus it will take
$\Omega(n^2)$ iterations until $\varepsilon = (1/2)^{\Theta(n)}$ is reached. 
\item[(3)] \emph{Critical vectors must be long.}
Why should we expect that critical vectors must be long? Intuitively, if $\varepsilon$ is getting smaller, then $x^*$ is moving closer to the hyperplane defined by $c$
and the cone of objective functions that are optimal at $x^*$ becomes very narrow. As a consequence, the length of critical vectors
should increase as $\varepsilon$ decreases. 

Recall that we termed $\tilde{c} \in \setZ^n$ critical if and only if
\[
  \max\{ \tilde{c}x \mid x \in P_I \} \leq \tilde{c}x^*.
\]
One of our key lemmas is to show that under some mild conditions, the 
left hand side can be lowerbounded by $\frac{1}{2}\|\tilde{c}\|_1 + \Theta(\|\tilde{c} - \frac{c}{\lambda}\|_1)$,
where $\lambda > 0$ is some scalar. As we will see, an immediate consequence is that for a critical vector $\tilde{c}$ it is
a necessary condition that there is a $\lambda > 0$ with
\[
  \|\lambda\tilde{c} - c\|_1 \leq O( \varepsilon \|c\|_1 ).
\]
In other words, it is necessary that $\tilde{c}$, if suitably scaled, well approximates
the vector $c$. In fact, this problem is well studied under the name
\emph{simultaneous Diophantine approximation}. Thus, if we want to show that critical vectors must be long, 
 it suffices to find a 
vector $c$ that does not admit good approximations using short vectors $\tilde{c}$. 
The simple solution is
to pick $c$ \emph{at random} from a suitable range; then $\|\lambda\tilde{c} - c\|_1$ will 
be large with high probability for all $\lambda$ and all short $\tilde{c}$.
\end{enumerate}

\section{A general strategy to lower bound the Chvátal rank}

%
We focus now on the polytope $P := P(c,\varepsilon)$ defined above and properties of 
critical vectors. 
We want to define $L_c(\varepsilon)$ as the $\|\cdot\|_1$-length of the shortest vector, that is
$x^*(\varepsilon)$-critical. Formally, let
\begin{eqnarray*}
L_{c}(\varepsilon) &:=& \min_{\tilde{c} \in \setZ_{\geq0}^n}\Big\{ \|\tilde{c}\|_1 \mid \tilde{c}x^* \geq \max_{x \in P_I} \tilde{c}x\Big\} \\
&=& \min_{\tilde{c} \in \setZ_{\geq0}^n}\Big\{ \|\tilde{c}\|_1 \mid \|\tilde{c}\|_1 \cdot \left(\frac{1}{2} + \varepsilon\right) \geq \max_{x \in P_I} \tilde{c}x\Big\}
\end{eqnarray*}
By definition, the function $L$ is monotonically non-increasing in $\varepsilon$ and 
$L_c(0) \leq \|c\|_1$. 

For example, if $c = (1,\ldots,1)$, it is not difficult to
show that $L_c(\varepsilon) \geq \frac{n}{2}$ for all $0<\varepsilon<\frac{1}{2}$ (see Appendix~\ref{sec:CriticalVectorsForAllOnesVector}). In fact, for all $c$ and $\varepsilon$, one can show a general 
\emph{upper} bound of $L_c(\varepsilon) \leq \frac{n}{\varepsilon}$ (see Appendix \ref{sec:UpperBoundOnLcEpsilon}). Later we will see that for some
choice of $c$ 
this bound is essentially tight --- for a long range of $\varepsilon$, and 
this will be crucial to prove our result.  

Observe that, in the definition of $L_c(\varepsilon)$, we only admit non-negative entries for $\tilde{c}$.
But it is not difficult to prove that since $P$ is a \emph{monotone} polytope (that is, 
$x \in P$, $\mathbf{0} \leq y\leq x \Longrightarrow y \in P$), the 
shortest critical vectors will be non-negative. 

\begin{lemma} \label{lem:ThereIsAshorterTightVector}
Let $\tilde{c} \in \setZ^n$ be $x^*$-critical. 
Then also the vector $\tilde{c}^+ \in \setZ_{\geq 0}^n$ with $\tilde{c}^+_i := \max\{ \tilde{c}_i, 0\}$ is $x^*$-critical.
Moreover $\|\tilde{c}^+\|_1 \leq \|\tilde{c}\|_1$.
\end{lemma}
\begin{proof}
One has
\[
  \tilde{c}^+x^* \geq \tilde{c}x^* \stackrel{\tilde{c}\textrm{ critical}}{\geq} \max\{ \tilde{c}x \mid x \in P_I \} = \max\{ \tilde{c}^+x \mid x \in P_I\},
\]
thus $\tilde{c}^+$ is critical. Here we used for the last equality that 
the optimum solutions for both expressions $\max\{ \tilde{c}x \mid x \in P_I \}$ 
and $\max\{ \tilde{c}^+x \mid x \in P_I \}$ would w.l.o.g. have $x_i=0$ whenever $\tilde{c}_i<0$ (using the monotonicity of $P_I$). 
\end{proof}


How does the length of critical vectors relate to the Chvátal rank? 
The next lemma answers this question. In fact, one iteration of the Gomory Chvátal closure, reduces $\varepsilon$ by essentially $\frac{1}{L_c(\varepsilon)}$. 

\begin{lemma}  \label{lem:LowerBoundOnRkDependingOnL}
Suppose $L_c(\varepsilon) \geq \frac{\gamma}{\varepsilon}$ for all $\delta_1 \leq \varepsilon \leq \delta_0$ (with  $\gamma \geq 2$).
Then $\rk(P(c,\delta_0)) \geq \frac{\gamma}{2} \cdot \ln(\frac{\delta_0}{\delta_1})$.
\end{lemma}
\begin{proof}
Abbreviate  $P := P(\delta_0,c)$. To measure the progress of the Chvátal operator, 
consider $\varepsilon_i := \max\{ \varepsilon : x^*(\varepsilon) \in P^{(i)} \}$. Let $k$ be the index such that $\delta_0 = \varepsilon_0 \geq \varepsilon_1 \geq \ldots \geq \varepsilon_{k-1} \geq \delta_1 > \varepsilon_k$.
Clearly $\rk(P) \geq k$.

Consider a fixed $i \in \{ 0,\ldots,k-1\}$. We want to argue that the difference between consecutive $\varepsilon_i$'s is very small,
i.e. $\frac{\varepsilon_{i+1}}{\varepsilon_i} \geq 1 - \frac{1}{\gamma}$. So assume that $\varepsilon_i > \varepsilon_{i+1}$, otherwise there is nothing to show.
Let $\tilde{c}_ix \leq \lfloor\beta_i\rfloor$ be the Gomory Chvátal cutting plane that cuts furthest w.r.t. 
the line segment defined by $x^*(\varepsilon)$. In other words  $\tilde{c}_ix \leq \beta_i$ is feasible for $P^{(i)}$
with $\tilde{c}_i \in \setZ^n$ and $\tilde{c}_ix^*(\varepsilon_{i+1}) = \lfloor\beta_i\rfloor$ (similar to Figure~\ref{fig:PolytopeP}.(b)).
Since $\varepsilon_{i}>\varepsilon_{i+1}$, we have $\tilde{c}_i x^*(\varepsilon_i) > \lfloor\beta_i\rfloor$. Combining this with
the fact that $P(c,\varepsilon_i) \subseteq P^{(i)}$, we know that 
 $\tilde{c}_i$ is critical w.r.t. $x^*(\varepsilon_i)$ and by assumption $\|\tilde{c}_i\|_1 \geq L_c(\varepsilon_i) \geq \frac{\gamma}{\varepsilon_i}$.
Writing down what we obtained, we see that
\[
 1 \geq \underbrace{\tilde{c}_ix^*(\varepsilon_i)}_{\leq\beta_i} - \underbrace{\tilde{c}_ix^*(\varepsilon_{i+1})}_{=\lfloor\beta_i\rfloor} = \tilde{c}_i \cdot \mathbf{1} \cdot (\varepsilon_{i} - \varepsilon_{i+1}) = \|\tilde{c}_i\|_1 \cdot (\varepsilon_{i}-\varepsilon_{i+1}) 
\geq \frac{\gamma}{\varepsilon_i} \cdot (\varepsilon_{i} - \varepsilon_{i+1})
\]
which can be rearranged to $\frac{\varepsilon_{i+1}}{\varepsilon_i} \geq 1- \frac{1}{\gamma}$ as claimed. Finally,
\[
 \delta_1 > \varepsilon_k = \delta_0 \cdot \prod_{i=0}^{k-1} \frac{\varepsilon_{i+1}}{\varepsilon_i} \geq \delta_0\cdot \left(1- \frac{1}{\gamma}\right)^{k} \geq \delta_0 \cdot e^{-2k/\gamma}
\]
using that $1-x\geq e^{-2x}$ for $0\leq x\leq \frac{1}{2}$. Rearranging yields  $k \geq \frac{\gamma}{2} \ln(\frac{\delta_0}{\delta_1})$.
\end{proof}

\section{Constructing a good Knapsack solution}
In order to provide a lower bound on $L_c(\varepsilon)$,
we  inspect the knapsack problem $\max\{ \tilde{c}x \mid x \in P_I\}$
for a critical vector $\tilde c$.
The crucial ingredient for our proof is to find a fairly tight lower bound on this quantity.


In the following key lemma (Lemma~\ref{lem:LowerBoundOnMaxCtildeOverPI}), we are going 
to show that (under some conditions on $c$) we can derive the lower bound:
$\max\left\{ \tilde{c}x \mid x \in P_I \right\} \geq \frac{1}{2}\|\tilde{c}\|_1 + \Omega\big( \left\|\tilde{c} - \frac{c}{\lambda}\right\|_1\big)$ for some $\lambda > 0$. Intuitively 
the vector $x=(\frac{1}{2}, \dots, \frac{1}{2})$
is already a (fractional) solution to the above knapsack problem of value $\|\tilde{c}\|_1/2$, 
but if $c$ and $\tilde{c}$ have a large angle, than one actually improve over that 
solution; in fact one can improve by the ``difference'' $\|\tilde{c} - \frac{c}{\lambda}\|_1$. 
Before the formal proof, let us describe, how to derive this lower bound
in an ideal world that's free of technicalities. 

Sort the items by their profit over cost ratio so that 
 $\frac{\tilde{c}_1}{c_1} \geq \ldots \geq \frac{\tilde{c}_n}{c_n}$. 
Since we are dealing with a knapsack problem, we start taking
the items with the best ratio into our solution. Suppose for
the sake of simplicity that we are lucky and the  $k$ 
items with largest ratio fit perfectly into the knapsack, i.e. $\sum_{i=1}^k c_i = \|c\|_1/2$. 
Then $J := [k]$ must actually be an \emph{optimum} knapsack solution.
Next, choose $\lambda >0$ such that $\frac{1}{\lambda}$ is the profit threshold, i.e. 
 $\frac{\tilde{c}_1}{c_1} \geq \ldots \geq \frac{\tilde{c}_k}{c_k} \geq \frac{1}{\lambda} \geq \frac{\tilde{c}_{k+1}}{c_{k+1}} \geq \ldots \geq	 \frac{\tilde{c}_n}{c_n}$.
Using that $\sum_{i \in J} c_i = \sum_{i \notin J} c_i$, we can express the profit of our solution as
\[
\sum_{i \in J} \tilde{c}_i = 
\frac{1}{2}\|\tilde{c}\|_1 + \frac{1}{2}\sum_{i \in J} \underbrace{\left(\tilde{c}_i-\frac{c_i}{\lambda}\right)}_{\geq 0} - \frac{1}{2}\sum_{i \notin J} \underbrace{\left(\tilde{c}_i-\frac{c_i}{\lambda}\right)}_{\leq 0} 
=  \frac{1}{2} \|\tilde{c}\|_1 + \frac{1}{2}\left\|\tilde{c} - \frac{c}{\lambda}\right\|_1
\]
proving the claimed lower bound on $\max\left\{ \tilde{c}x \mid x \in P_I \right\}$. In a non-ideal world, the greedily obtained solution
would not perfectly fill the knapsack, i.e. $\sum_{i=1}^k c_i < \|c\|_1/2$. 
To fill this gap, we rely on the concept of \emph{additive basis}.

\begin{definition}
Let $I = [a,b] \cap \setZ_{\geq0}$ be an interval of integers. We call subset $B \subseteq \setZ_{\geq0}$  
an additive basis for $I$ if for every $k \in I$, there are numbers $S \subseteq B$
such that
\[
  \sum_{s \in S} s = k.
\]
\end{definition}
In other words, we can express every number in $I$ as a sum of numbers in $B$.
For example $\{2^0,2^1,2^2,\ldots,2^k\}$ is an additive basis for $\{0,\ldots,2^0+2^1 + \ldots + 2^k\}$.
The geometric consequence for a knapsack polytope $Q = \{ x \in \setR_{\geq0}^n \mid cx \leq \|c\|_1/2 \}$
is the following: 
if $c_1,\ldots,c_n$ are integral numbers that contain an additive basis (with at most $n/2$ elements) for $\{ 0,\ldots,\|c\|_{\infty}\}$
and, let's say $\|c\|_{\infty} \leq O(\frac{\|c\|_1}{n})$, then the face $cx = \|c\|_1/2$ contains
$2^{\Omega(n)}$ many $0/1$ points. The reason for this fact is that we can extend any
subset of items $I \subseteq [n]$ that does not exceed the capacity and that does not
contain any basis element, to a solution that fully fills the knapsack.
In the following, we abbreviate as usual $c(J) := \sum_{i \in J} c_i$.


\begin{lemma} \label{lem:LowerBoundOnMaxCtildeOverPI}
Let $c \in \setZ_{>0}^n$, $\tilde{c} \in \setR_{>0}^n$ and 3 disjoint index sets $B_1,B_2,B_3 \subseteq [n]$ 
such that each set $\{ c_i \mid i \in B_{\ell}\}$ is an additive basis for the interval  $I = \{ 0,\ldots,\|c\|_{\infty}\}$
with $\|c\|_{\infty} \leq \delta\|c\|_1$ and $c(B_{\ell}) \leq \delta\cdot\|c\|_1$ for all $\ell=1,2,3$ with $\delta := \frac{1}{100}$. Then there is a scalar $\lambda := \lambda(c,\tilde{c}) > 0$ such that
\[
  \max\left\{ \tilde{c}x \mid x \in \{ 0,1\}^n; \; cx \leq \frac{\|c\|_1}{2}\right\} \geq \frac{1}{2}\|\tilde{c}\|_1 + \frac{1}{16}\cdot \left\|\tilde{c} - \frac{c}{\lambda}\right\|_1
\]
\end{lemma}
\begin{proof}
Since we allow $\tilde{c}_i \in \setR$, there lies no harm in perturbing the coefficients slightly such that 
the profit/cost ratios $\frac{\tilde{c}_i}{c_i}$ are pairwise distinct. 
We sort the indices such that $\frac{\tilde{c}_1}{c_1} > \ldots > \frac{\tilde{c}_n}{c_n}$. 
Choose $\lambda>0$ such that 
\[
\sum_{i:\tilde{c}_i/c_i > 1/\lambda} c_i \in \left[\frac{\|c\|_1}{2}-\|c\|_{\infty},\frac{\|c\|_1}{2}\right]
\]
and there is no $i$ with $\frac{\tilde{c}_i}{c_i} = \frac{1}{\lambda}$
(recall that $\frac{\tilde{c}_i}{c_i} > \frac{1}{\lambda} \Leftrightarrow \lambda\tilde{c}_i-c_i>0$). In other words, $\frac{1}{\lambda}$ is a \emph{profit threshold}
and ideally we would like to construct a solution for our knapsack problem by selecting the items above the threshold.
Let $q \in \{ 1,\ldots,n\}$ be the number such that $\frac{\tilde{c}_i}{c_i} > \frac{1}{\lambda} \Leftrightarrow i\leq q$, i.e. 
 $\frac{\tilde{c}_1}{c_1} > \ldots > \frac{\tilde{c}_q}{c_q} > \frac{1}{\lambda} > \frac{\tilde{c}_{q+1}}{c_{q+1}} > \ldots > \frac{\tilde{c}_n}{c_n}$.
For every item $i$ we define the \emph{relative profit} $w_i := \tilde{c}_i - \frac{c_i}{\lambda}$. 
Note that  $\frac{w_i}{c_i} = \frac{\tilde{c}_i}{c_i} - \frac{1}{\lambda}$ and $w_i>0 \Leftrightarrow i\leq q$, but the values $w_i$ are not necessarily monotonically decreasing.
The way how we defined $w$ yields $\|w\|_1 = \|\tilde{c} - \frac{c}{\lambda}\|_1$. 
Since the $B_{\ell}$'s are disjoint, one has $\sum_{\ell = 1}^3 \sum_{i \in B_{\ell}} |w_i| \leq \| w\|_1$. Thus we can pick 
one of the sets $B := B_{\ell}$ such that $\sum_{i \in B} |w_i| \leq \frac{1}{3} \| w\|_1$.

We are now going to construct a knapsack solution that fully fills the knapsack.
Let $k \in [n]$ maximal be such that 
\[
   \sum_{i \in \{ 1,\ldots,k\} \backslash B} c_i \leq \frac{\|c\|_1}{2} 
\]
In other words, if we take items $\{1,\ldots,k\} \backslash B$ into our knapsack, we have capacity at most $\|c\|_{\infty}$ left.
Next, construct an arbitrary solution $J' \subseteq B$ that perfectly fills the
remaining capacity, i.e. 
for $J := (\{ 1,\ldots,k\} \backslash B) \cup J'$ we have $c(J) = \frac{\|c\|_1}{2}$.
Observe that 
\begin{equation}  \label{eq:ckUpperBound}
c([k]) \leq \underbrace{c([k]\backslash B)}_{\leq \|c\|_1/2} + \underbrace{c(B)}_{\leq \delta\|c\|_1} \leq \left(\frac{1}{2} + \delta\right)\cdot \|c\|_1.
\end{equation}
Moreover,
\begin{equation}  \label{eq:ckLowerBound}
  c([k]) \geq c([k]\backslash B) \geq \frac{\|c\|_1}{2} - \|c\|_{\infty} \geq \left(\frac{1}{2} - \delta\right)\cdot \|c\|_1  
\end{equation}

\begin{figure}
\begin{center}
\psset{xunit=0.7cm,yunit=0.3cm}
\begin{pspicture}(2,-2)(20,17)
\psaxes[arrowsize=6pt,labels=none,ticks=none]{->}(0,0)(0,0)(22.7,15.5) \rput[r](-10pt,15){$\frac{\tilde{c}_i}{c_i}$}
\drawRect{fillstyle=solid,fillcolor=lightgray}{0}{0}{2}{14.5}
\drawRect{fillstyle=vlines,dashcolor=gray}{2}{0}{1}{14}
\drawRect{fillstyle=solid,fillcolor=lightgray}{3}{0}{1}{13}
\drawRect{fillstyle=vlines,dashcolor=gray}{4}{0}{2}{12}
\drawRect{fillstyle=solid,fillcolor=lightgray}{6}{0}{1}{11}
\drawRect{fillstyle=solid,fillcolor=lightgray}{7}{0}{2}{10}
\drawRect{fillstyle=solid,fillcolor=lightgray}{9}{0}{1}{9}
\drawRect{fillstyle=solid,fillcolor=lightgray}{10}{0}{2}{8}
\drawRect{}{12}{0}{1}{7}
\drawRect{}{13}{0}{2}{6}
\drawRect{}{15}{0}{1}{5}
\drawRect{fillstyle=none}{16}{0}{1}{4}
\drawRect{}{17}{0}{3}{3}
\drawRect{fillstyle=vlines*,dashcolor=gray,fillcolor=lightgray}{20}{0}{2}{2}
\psline(22,0)(22,-8pt) \rput[c](22,-18pt){$\|c\|_1$}
\psline(11,0)(11,-8pt) \rput[c](11,-18pt){$\frac{1}{2}\|c\|_1$}
\psline[linestyle=dashed](-5pt,8.5)(22,8.5) \rput[r](-10pt,8.5){$\frac{1}{\lambda}$}
\pnode(0,-2){LC1} \pnode(2,-2){RC1} \ncline[arrowsize=6pt]{|<->|}{LC1}{RC1} \nbput{$c_1$} 
\pnode(2,-2){LC2} \pnode(3,-2){RC2} \ncline[arrowsize=6pt]{|<->|}{LC2}{RC2} \nbput{$c_2$} 
\pnode(2.5,13.5){LB1} \pnode(3.5,15){LB2} \ncline[arrowsize=6pt]{->}{LB2}{LB1} \nput[labelsep=2pt]{0}{LB2}{$\in B$}
\pnode(6.5,10.5){LJ1} \pnode(7,12){LJ2} \ncline[arrowsize=6pt]{->}{LJ2}{LJ1} \nput{90}{LJ2}{$\in J$}
\pnode(21,1.5){LBJ1} \pnode(21,3){LBJ2} \ncline[arrowsize=6pt]{->}{LBJ2}{LBJ1} \nput{90}{LBJ2}{$\in B\cap J$}
\pnode(11,7.5){LitemK1} \pnode(12,9){LitemK2} \ncline[arrowsize=6pt]{->}{LitemK2}{LitemK1} \nput{90}{LitemK2}{item $k$}
\pnode(9.5,8.75){LitemQ1} \pnode(10,10){LitemQ2} \ncline[arrowsize=6pt]{->}{LitemQ2}{LitemQ1} \nput{90}{LitemQ2}{item $q$}
\pnode(0,16){LK1} \pnode(2,16){LK2} \pnode(3,16){LK3} \pnode(4,16){LK4} \pnode(6,16){LK5} \pnode(12,16){LK6} \ncline[arrowsize=6pt]{|<-}{LK1}{LK2} \ncline[arrowsize=6pt,linestyle=dotted]{-}{LK2}{LK3} \ncline[arrowsize=6pt]{-}{LK3}{LK4} \ncline[arrowsize=6pt,linestyle=dotted]{-}{LK4}{LK5} \ncline[arrowsize=6pt]{->|}{LK5}{LK6}  \naput{$c([k]\backslash B)\in\big[\frac{\|c\|_1}{2} - \|c\|_{\infty}, \frac{\|c\|_1}{2}\big] \hspace{2cm}$}
\psline[linewidth=1.5pt](6,13)(6,-8pt) \rput[c](6,-15pt){$(\frac{1}{2}-\delta)\|c\|_1$}
\psline[linewidth=1.5pt](16,13)(16,-8pt) \rput[c](16,-15pt){$(\frac{1}{2}+\delta)\|c\|_1$}
\end{pspicture}
\caption{Visualization of the construction of $J$: Take items with best profit/cost ratio (skipping items in the basis $B$) as long as possible. Then fill the remaining gap with arbitrary items from $B$. \label{fig:ConstructionOfJ}}
\end{center}
\end{figure}
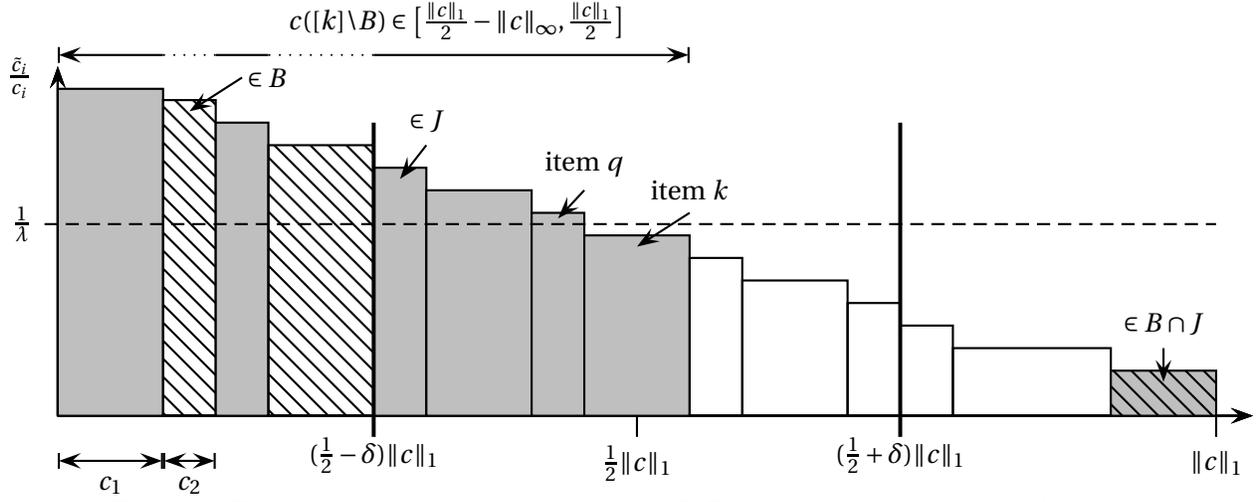

We call an item $i$ \emph{central} if $(\frac{1}{2}-\delta)\|c\|_1 \leq c([i]) \leq (\frac{1}{2}+\delta)$.
We cannot be sure appriori whether central items are selected into $J$ or not. However, we can 
prove that due to the sorting they have a small $|w_i|$-value
anyway.
Let us abbreviate $W^+ := \sum_{i \leq q} w_i$ and $W^- := \sum_{i>q} |w_i|$ (so that $\|w\|_1 = W^+ + W^-$).

\begin{claim} \label{claim:WeightOfCentralWindow}
$\sum_{i : \; (\frac{1}{2}-\delta)\|c\|_1 \leq c([i]) \leq (\frac{1}{2}+\delta)\|c\|_1} |w_i| \leq 9\delta \|w\|_1.$
\end{claim}
\begin{proofofclaim}
We abbreviate $I_+ := \{ i \mid w_i > 0\}$ and $I_- := \{ i \mid w_i < 0\}$. Furthermore 
$I_+^{\delta} := \{ i \in I_+ \mid c([i]) \geq (\frac{1}{2}-\delta)\|c\|_1 \}$ and $I_-^{\delta} := \{ i \in I_- \mid c([i]) \leq (\frac{1}{2}+\delta)\|c\|_1 \}$. 
Note that $c(I_+),c(I_-) \geq \frac{1}{2}\|c\|_1-2\|c\|_{\infty} \geq \frac{1}{3}\|c\|_1$ (since $\|c\|_{1} \geq 6\|c\|_{\infty}$)
and $c(I_+^{\delta}),c(I_-^{\delta}) \leq \delta\|c\|_1 + 2\|c\|_{\infty} \leq 3\delta\|c\|_1$ (since $\|c\|_{\infty} \leq \delta\|c\|_1$).

Recall that the items are sorted such that the values $\frac{w_i}{c_i} = \frac{\tilde{c}_i}{c_i} - \frac{1}{\lambda}$
decrease and $I_+^{\delta}$ is a set of maximal indices within $I_+$, thus the average of $\frac{w_i}{c_i}$
over items in $I_+^{\delta}$ cannot be higher than the average over $I_+$. Formally
$\frac{w(I_{+}^{\delta})}{c(I_{+}^{\delta})} \leq  \frac{W^+}{c(I_+)}$,
thus 
\begin{equation} \label{eq:WIplusDeltaBound}
w(I_{+}^{\delta}) \leq \frac{c(I_{+}^{\delta})}{c(I_+)} \cdot W^+ \leq \frac{3\delta \|c\|_1}{ \|c\|_1/3} W^+ = 9\delta \cdot W^+.
\end{equation}
 Analogously
$  \frac{\sum_{i \in I_{-}^{\delta}} |w_i|}{c(I_-^{\delta})} 
\leq \frac{W^-}{c(I_-)}$,
and hence
\begin{equation} \label{eq:WIminusDeltaBound}
  \sum_{i \in I_{-}^{\delta}} |w_i| \leq  \frac{c(I_-^{\delta})}{c(I_-)} \cdot W^- \leq \frac{3\delta\|c\|_1}{\|c\|_1/3}\cdot W^- \leq 9\delta \cdot W^-.
\end{equation}
Adding up \eqref{eq:WIplusDeltaBound} and \eqref{eq:WIminusDeltaBound} yields the claim $\sum_{i \in I_{+}^{\delta} \cup I_{-}^{\delta}} |w_i| \leq 9\delta \cdot (W^+ + W^-) = 9\delta \|w\|_1$.
\end{proofofclaim}

\begin{claim} \label{claim:wBound}
$w(J) - w([n] \backslash J) \geq \frac{1}{8}\cdot \|w\|_1$.
\end{claim}
\begin{proofofclaim}
We call an index $i \in [n]$ \emph{correct}, if $i \in J \Leftrightarrow w_i>0$.
In other words, indices with $w_i>0$ that are in $J$ are correct and indices
with $w_i<0$ and $i \notin J$ are correct -- all other indices are \emph{incorrect}.
A correct index $i$ contributes $+|w_i|$ to the sum $w(J) - w([n] \backslash J)$ and an
incorrect index contributes $-|w_i|$. Thus if all indices would be correct, we would 
have $w(J) - w([n] \backslash J) = \|w\|_1$. From this amount, we want to deduct
contributions for incorrect indices. An index can either be incorrect
if it is in $B$ (for those we have $\sum_{i \in B} |w_i| \leq \frac{1}{3}\|w\|_1$) or if
it lies in the central window, i.e. $c([i]) \in (\frac{1}{2} ± \delta)\|c\|_1$
(for those items we have $\sum_{i : \; (\frac{1}{2}-\delta)\|c\|_1 \leq c([i]) \leq (\frac{1}{2}+\delta)\|c\|_1} |w_i| \leq 9\delta \|w\|_1$ according to Claim~\ref{claim:WeightOfCentralWindow}).
Subtracting these quantities, we obtain
\begin{eqnarray}
  \sum_{i \in J} w_i - \sum_{i \notin J} w_i
&\geq& \left(1-2\cdot9\delta - 2\cdot\frac{1}{3}\right)\cdot \|w\|_1 \geq \frac{1}{8}\|w\|_1 \label{eq:generalWBound}
\end{eqnarray}
for $\delta = \frac{1}{100}$.
\end{proofofclaim}

Finally, we note that the vector $\tilde{x} \in \{ 0,1\}^n$ with $\tilde{x}_i := 1$ if $i \in J$ and $0$ otherwise, 
satisfies the claim.
\begin{eqnarray*}
  \tilde{c}\tilde{x} &=& \frac{1}{2}\|\tilde{c}\|_1 + \frac{1}{2}\sum_{i \in J} \tilde{c}_i - \frac{1}{2}\sum_{i \notin J} \tilde{c}_i \\
&\stackrel{\sum_{i \in J} c_i = \sum_{i \notin J} c_i}{=}& \frac{1}{2}\|\tilde{c}\|_1 + \frac{1}{2}\sum_{i \in J} \left(\tilde{c}_i-\frac{c_i}{\lambda}\right) - \frac{1}{2}\sum_{i \notin J} \left(\tilde{c}_i-\frac{c_i}{\lambda}\right) \\
&=& \frac{1}{2} \|\tilde{c}\|_1 + \frac{1}{2} \sum_{i \in J} w_i - \frac{1}{2} \sum_{i \notin J} w_i \\
&\stackrel{\textrm{Claim~\eqref{claim:wBound}}}{\geq}&  \frac{1}{2} \|\tilde{c}\|_1 + \frac{1}{16}\|w\|_1 \\
&=&  \frac{1}{2} \|\tilde{c}\|_1 + \frac{1}{16}\Big\|\tilde{c} - \frac{c}{\lambda}\Big\|_1
\end{eqnarray*}
Here we use that $\sum_{i \in J} c_i = \sum_{i \notin J} c_i$.
\end{proof}

Now, we can get a very handy necessary condition on critical vectors.
Namely, if the conditions on $c$ (see Lemma~\ref{lem:LowerBoundOnMaxCtildeOverPI})
are satisfied, then any critical vector must have $\|\lambda\tilde{c} - c\|_1 \leq O(\varepsilon)\cdot\|c\|_1$.
To prove that critical vectors must be long, 
it remains to find a vector $c$ such that $\|\lambda\tilde{c} - c\|_1$ is
large for all short vectors $\tilde{c}$. 

\section{Random normal vectors}

In this section, we will see now, that a random vector $a$ cannot be well approximated by short vectors; later this vector $a$ will be essentially the first half of 
the normal vector $c$. 
In the following, for any vector $a \in \setR^m $, and any index subset $I \subseteq [m]$, we let $(a)_I \in \setR^{|I|}$ be the vector $(a_i, i\in I)$.
For $D := 2^{m/8}$, pick $a_1,\ldots,a_m \in \{D,\ldots,2D\}$ uniformly and independently at random. 

We first informally describe, why this random vector $a$ is hard to approximate
with high probability. Let us fix values of $\lambda$ and $\varepsilon$ and call an index $i$ \emph{good}, if there is an integer $\tilde{a}_i \in \{ 0,\ldots,o(\frac{1}{\varepsilon})\}$ such that $|\lambda\tilde{a}_i - a_i| \leq O(\varepsilon D)$.
Since we choose $a_i$ from $D$ many possible choices, we have $\Pr[i\textrm{ good}] \leq o(\frac{1}{\varepsilon}) \cdot O(\varepsilon D) \cdot \frac{1}{D} = o(1)$.
For the event ``$\exists\tilde{a}: \|\tilde{a}\|_1 \leq o(\frac{m}{\varepsilon})\textrm{ and } \|\lambda\tilde{a} - a\|_1 \leq O(m\varepsilon D)$'' one needs $\Omega(m)$ many good indices and by standard arguments 
the probability for this to happen is $o(1)^m$. Finally we can argue that the number of
distinct values of $\varepsilon$ and $\lambda$ that need to be considered is $2^{O(m)}$. Thus by the union 
bound, the probability that there are \emph{any} $\varepsilon$, $\lambda$ and $\tilde{a} \in \{ 0,\ldots,o(\frac{m}{\varepsilon})\}^m$
with $\|\lambda\tilde{a}_i - a\|_1 \leq O(\varepsilon mD) = O( \varepsilon \|a\|_1)$ is still upper bounded by $o(1)^m$.

We will now give a formal argument.

\begin{lemma} \label{lem:randomknapsack} 
There is a constant $\alpha > 0$ such that for $m$ large enough, 
\begin{equation} \label{eq:randomknapsack}
  \Pr\Big[\exists (\varepsilon,\lambda,\tilde{a}) \in [\tfrac{1}{D},\tfrac{1}{4}] × \setR_{>0} × \setZ^m  : \|\tilde{a}\|_1 \leq \frac{m}{\alpha \varepsilon}\textrm{ and } \|\lambda\tilde{a}-a\|_1 \leq 100\varepsilon m\cdot D\Big] \leq \left(\frac{1}{2}\right)^{m}  
\end{equation}
\end{lemma}
\begin{proof}
We want to bound the above probability in \eqref{eq:randomknapsack} by using the union bound over 
all $\lambda > 0$ and all $\varepsilon > 0$. 
First of all, $\|\lambda\tilde{a}-a\|_1$ 
is a  piecewise linear function in $\lambda$. Therefore,  
we can restrict our attention to the values $\lambda = \frac{a_i}{\tilde a_i}$
for some $a_i \in \{D, \dots, 2D \}$ and ${\tilde a_i \in \{0, \dots, \frac{m}{\alpha \varepsilon} \}}$. That is, assuming $\varepsilon \geq \frac{1}{D}$, the number of different $\lambda$ values that really matter are 
bounded by $(D+1) \cdot (\frac{m}{\alpha \varepsilon} + 1) \leq (2D)^3$. 
Moreover, we only need to consider those $\varepsilon$, 
where at least one of the bounds $\|\tilde{a}\|_1 \leq \frac{m}{\alpha \varepsilon}$ or 
$\|\lambda\tilde{a}-a\|_1 \leq  100\varepsilon m\cdot D$ is tight\footnote{The reason is that $\varepsilon$ and $\varepsilon'$ with
$\lfloor m/(\alpha\varepsilon)\rfloor=\lfloor m/(\alpha\varepsilon')\rfloor$ and $\lfloor100\varepsilon mD\rfloor=\lfloor100\varepsilon'mD\rfloor$ belong to identical events.}. But $\|\tilde{a}\|_1$ attains at most
$2mD\leq(2D)^2$ many values and $ \|\lambda\tilde{a}-a\|_1$ attains at most $(2D)^4$ many values.
In total the number of relevant values of pairs $(\lambda,\varepsilon)$ is bounded by $(2D)^9 \leq 2^{2m}$.
Thus, by the union bound it suffices to prove that for every \emph{fixed} pair 
$\lambda > 0$ and $\varepsilon $, one has
\begin{equation} \label{eq:Prob1NormOfAsmallForAllShortATilde}
 \Pr\Big[\exists \tilde{a} \in \setZ_{\geq0}^m : \|\tilde{a}\|_1 \leq \frac{m}{\alpha\varepsilon}\textrm{ and } \|\lambda\tilde{a}-a\|_1 \leq 100\varepsilon m\cdot D\Big] \leq 2^{-3m}
\end{equation}
for $\alpha>0$ large enough.
Note that, for any vector $\tilde a \in \setZ_{\geq0}^m : \|\tilde{a}\|_1 \leq \frac{m}{\alpha\varepsilon}$ there exists a subset of indices $I \subset [m]$ with $|I| \geq \frac{m}{2}$ such that $\|(\tilde a)_I\|_\infty \leq  \frac{2}{\alpha \varepsilon}$. If not, then $\|\tilde{a}\|_1 > \frac{m}{2}\cdot \frac{2}{\alpha \varepsilon}$  leading to a contradiction. 
%
%
%
Similarly, we can say that there exists a subset of indices $J \subseteq I$, with 
$|J| \geq |I|/2$ such that $\|(\lambda\tilde{a}-a)_J\|_\infty \leq 400 \varepsilon \cdot D$. If not, then $\|(\lambda\tilde{a}-a)_I\|_1 > \frac{m}{4}\cdot 400 \varepsilon \cdot D$ again leading to a contradiction. It follows that the left hand side of \eqref{eq:Prob1NormOfAsmallForAllShortATilde}
is bounded by
%
\begin{equation} \label{eq:ProbInfNormOfAsmallForShortAtilde} 
\Pr\Big[\exists \tilde{a} \in \setZ_{\geq0}^m \textrm{ and } J \subseteq [m], 
|J| = \frac{m}{4} : \|(\tilde{a})_J\|_\infty \leq \frac{2}{\alpha\varepsilon}\textrm{ and } \|(\lambda\tilde{a}-a)_J\|_\infty \leq 400 \varepsilon \cdot D\Big] 
\end{equation}
For a fixed index $i \in [m]$, we have 
\begin{eqnarray*}
 \Pr\Big[\exists \tilde{a}_i \in \setZ: \tilde{a}_i \leq \frac{2}{\alpha\varepsilon}\textrm{ and } |\lambda\tilde{a}_i-a_i| \leq 400 \varepsilon \cdot D\Big] &\leq& \frac{1}{D} \sum_{\tilde{a}_i=0}^{2/(\alpha\varepsilon)} \left| \setZ \cap [\lambda\tilde{a}_i - 400\varepsilon D, \lambda\tilde{a}_i + 400\varepsilon D]\right| \\
&\leq& \left(\frac{2}{\alpha\varepsilon}+1\right) \cdot \frac{800 \varepsilon D+1}{D} \leq \frac{1800}{\alpha}. 
\end{eqnarray*}
Here, we use that  $\varepsilon \geq \frac{1}{D}$ and every number $\lambda\tilde{a}_i$ is at distance $400\varepsilon D$ to at most $800\varepsilon D+1$ many integers. 
Moreover, we upperbound the number of all different index subsets of cardinality $m/4$ by $2^m$. It follows that \eqref{eq:ProbInfNormOfAsmallForShortAtilde} can be bounded by
%
%
\[
 2^m \cdot \Big(\frac{1800}{\alpha}\Big)^{m/4} 
\leq \Big( \frac{1}{2}\Big)^{3m}\]
for $\alpha > 0$ large enough.
%


\end{proof}

\section{A $\Omega(n^2)$ bound on the Chvátal rank}

Now we have all tools together, to obtain a quadratic lower bound on the
 Chvátal rank of a $0/1$ polytope.
%
Let $m := \frac{n}{2}$ and let $a$ be an $m$-dimensional vector according to 
Lemma~\ref{lem:randomknapsack} (i.e. a vector satisfying the
event in \eqref{eq:randomknapsack}). 
Let $b = (2^0,2^1,2^2,\ldots,2^{m/8+1})$ be a basis for $\{0,\ldots,2D\}$ (recall that $D = 2^{m/8}$). 
We choose $c := (a,b,b,b,\mathbf{0}) \in \setZ_{\geq 0}^n$ (note that $m + 3\cdot(\frac{m}{8}+2) \leq n$,
so that we can indeed fill the vector $c$ with zero's to obtain $n$ many entries).


\begin{theorem} \label{thm:ChvatalRankAtLeastQuadratic}
The Chvátal rank of $P := P(c,\frac{1}{4})$ is $\Omega(n^2)$.
\end{theorem}

\begin{proof}
By Lemma~\ref{lem:LowerBoundOnRkDependingOnL}, the statement follows if we show 
that for all $\frac{1}{D} \leq \varepsilon \leq \frac{1}{32}$ one has $L_c(\varepsilon) \geq \Omega(\frac{n}{\varepsilon} )$.

Hence, fix an $\varepsilon$ and let $\tilde{c}$ be the $x^*(\varepsilon)$-critical vector
with minimal $\|\tilde{c}\|_1$.
Obviously, $c$ contains 3 disjoint bases for  the interval $\{0, \dots, \|c\|_\infty\}$.
Moreover: 
\[\|c\|_\infty \leq \|b\|_1 \leq 4D \stackrel{n \textrm{ large enough }}{\leq} 
\frac{1}{100} \|c\|_1. \]
Therefore, we can apply Lemma \ref{lem:LowerBoundOnMaxCtildeOverPI}  to obtain 
\[
\Big(\frac{1}{2}+\varepsilon\Big)\|\tilde{c}\|_1 = \tilde{c}x^*(\varepsilon) \stackrel{\tilde{c}\textrm{ critical}}{\geq} \max\left\{ \tilde{c}x \mid x \in \{ 0,1\}^n; \; cx \leq \frac{\|c\|_1}{2}\right\} \stackrel{\textrm{Lem.~\ref{lem:LowerBoundOnMaxCtildeOverPI}}}{\geq} \frac{1}{2}\|\tilde{c}\|_1 + \frac{1}{16}\cdot \left\|\tilde{c} - \frac{c}{\lambda}\right\|_1
\]
Subtracting $\frac{1}{2}\|\tilde{c}\|_1$ from both sides and multiplying with $\lambda > 0$ 
yields $\frac{1}{16}\|\lambda\tilde{c} - c\|_1 \leq \varepsilon\|\lambda\tilde{c}\|_1$. 
We claim that $\|\lambda\tilde{c}\|_1 \leq 2\|c\|_1$, since otherwise by the reverse triangle
inequality
 $\|\lambda\tilde{c} - c\|_1 \geq \|\lambda\tilde{c}\|_1 - \|c\|_1 > \frac{1}{2}\|\lambda\tilde{c}\|_1 \geq 16\varepsilon\|\lambda\tilde{c}\|_1$, 
which is a contradiction. Thus we have $\|\lambda\tilde{c} - c\|_1 \leq 32\varepsilon\|c\|_1$.
Now, let $\tilde{a}$ be the first $m$ entries of $\tilde{c}$,
then
 \[
\|\lambda\tilde{a} - a\|_1 \leq \|\lambda\tilde{c} - c\|_1 \leq 32\varepsilon\|c\|_1 \leq 64\varepsilon nD
\]
But inspecting again the properties of vector $a$ (see Lemma~\ref{lem:randomknapsack}), 
any such vector $\tilde{a}$ must have length $\|\tilde{a}\|_1 \geq \Omega(\frac{m}{\varepsilon})$. 
Since $m=n/2$, this implies $\|\tilde{c}\|_1 \geq \|\tilde{a}\|_1 \geq \Omega(\frac{n}{\varepsilon})$. 
Eventually, we apply Lemma~\ref{lem:LowerBoundOnRkDependingOnL} and obtain
that $\rk(P) \geq \Omega(n \cdot \log( \frac{1/32}{1/D})) = \Omega(n^2)$.

%
%
%
%
\end{proof}

We close the paper with a remark. A vector $d$ is called \emph{saturated} w.r.t. $P$ if it has an integrality gap of $1$, 
i.e. $ \max\{ dx \mid x \in P\} = \max\{ dx \mid x \in P_I\}$.
Of course, if $d \in \setZ^n$ is saturated, then the GC cut induced by $d$ does not cut off
any point, i.e. $GC_P(d) \cap P = P$.
With this definition, one could rephrase the statement of Theorem~\ref{thm:ChvatalRankAtLeastQuadratic}
as: The vector $c$ needs $\Omega(n^2)$ many iterations to be saturated. 
Note that \cite{ChvatalRankEisenbrandSchulzCombinatorica03} prove that any vector $c \in \setZ^n$
is saturated after $O(n^2 + n \log \|c\|_{\infty})$ many iterations, which gives the tight bound of $O(n^2)$ 
for our choice of $c$.

\subsection*{Acknowledgements}

The authors are very grateful to Michel X. Goemans for useful discussions.

\bibliographystyle{alpha}
\bibliography{chvatalrank}

\appendix

\newpage



\section{A general upper bound on $L_c(\varepsilon)$ \label{sec:UpperBoundOnLcEpsilon}}

Our quadratic lower bound on the Chvátal rank uses that
we can find a normal vector $c$ such that $L_c(\varepsilon) \geq \Omega(\frac{n}{\varepsilon})$
for a large range of $\varepsilon$. We want to show here that this bound is
asymptotically tight. 
\begin{lemma}  \label{eq:UpperBoundOnLcEpsilon}
For any $c \in \setZ_{\geq0}^n$ and any $0<\varepsilon<\frac{1}{2}$ we have $L_{c}(\varepsilon) \leq \frac{n}{\varepsilon}$.
\end{lemma}
\begin{proof}
We provide a simple choice for $\tilde{c}$. 
Let $\delta := \frac{n}{\|c\|_1 \cdot \varepsilon }$ be a scalar, then pick $\tilde{c} :=  \lfloor\delta \cdot c\rfloor \in \setZ_{\geq0}^n$. We have to verify that indeed $\tilde{c}x^* > \tilde{c}x$ for every
point $x \in P_I$. In fact, it is not difficult to prove, using 
that $cx^* \geq cx + \varepsilon\|c\|_1$ (since no point in $P_I$ is better 
than $(\frac{1}{2},\ldots,\frac{1}{2})$ for objective function $c$) and $\frac{1}{\delta}\tilde{c} \approx c$. In more detail, 
we have
\begin{eqnarray*}
  \tilde{c}x^* - \tilde{c}x 
 &\sim& \frac{1}{\delta}\tilde{c}x^* - \frac{1}{\delta}\tilde{c}x \\
&=& \underbrace{cx^* - cx^*}_{=0} + \frac{1}{\delta}\tilde{c}x^* - \frac{1}{\delta}\tilde{c}x + \underbrace{cx - cx}_{=0} \\
&=& \underbrace{cx^* - cx}_{\geq \varepsilon\|c\|_1\textrm{ since }x \in P_I}   - \underbrace{\underbrace{(x^*-x)}_{\in [-1,1]^n}\cdot \underbrace{(c - \frac{1}{\delta} \tilde{c})}_{\in [0,\frac{1}{\delta}]^n }}_{\leq n/\delta} \\ 
&\geq& \varepsilon \cdot \|c\|_1 - \frac{n}{\delta} \\ 
 &\stackrel{\delta = \frac{n}{\varepsilon\|c\|_1} }{=}& 0
\end{eqnarray*}
In other words, $\tilde{c}$ is indeed maximized at $x^*$ and 
our choice for $\tilde{c}$ was valid. Now let us inspect the
length of that vector: 
\[ 
  \|\tilde{c}\|_1 = \| \lfloor\delta c\rfloor\|_1 \leq \delta \|c\|_1 = \frac{n}{\varepsilon\|c\|_1} \cdot \|c\|_1 = \frac{n}{\varepsilon}.
\]
and the claim follows.
\end{proof}

\section{Critical vectors for $c=(1,\ldots,1)$  \label{sec:CriticalVectorsForAllOnesVector}}

\begin{lemma}
Let $c=(1,\ldots,1)$. Let $\tilde{c} \in \setZ^n$ be a vector that is $x^* = (\frac{1}{2}+\varepsilon,\ldots,\frac{1}{2} + \varepsilon)$-critical
with $0<\varepsilon<\frac{1}{2}$. Then $\|\tilde{c}\|_1 > \frac{n}{2}$.
\end{lemma}
\begin{proof}
Let $\tilde{c}$ be a vector such that $ \tilde{c}x^* \geq \max\{ \tilde{c}x \mid x \in P_I\}$. 
Assume for the sake of contradiction that $|\supp(\tilde{c})| \leq \frac{n}{2}$. 
It suffices to find a vector $\tilde{x} \in P_I$ so that $\tilde{c}\tilde{x} > \tilde{c}x^*$ and we have our contradiction.
We know from Lemma~\ref{lem:ThereIsAshorterTightVector}, that the shortest 
 $x^*$-critical vectors will be non-negative, so we assume  $\tilde{c} \geq \mathbf{0}$. 
Define $\tilde{x} \in \{ 0,1\}^n$ with
\[
\tilde{x}_i := \begin{cases} 1 & \tilde{c}_i\neq0 \\ 0 & \textrm{otherwise} \end{cases}
\]
Then $\|\tilde{x}\|_1 = |\supp(\tilde{c})| \leq \frac{n}{2}$ hence $\tilde{x} \in P_I$.
Moreover $\tilde{c}\tilde{x} = \|\tilde{c}\|_1$. On the other hand $\tilde{c}x^* = (\frac{1}{2} + \varepsilon) \|\tilde{c}\|_1$. 
Thus  $\tilde{c}$ does not attain the maximum at  $x^*$ and the claim follows by contradiction. 
\end{proof}

\end{document}